\documentclass[10pt]{amsart}

\usepackage{amsmath,amsthm,amsfonts,latexsym,amssymb,mathrsfs,setspace,enumerate}

\newtheorem{thm}{Theorem}

\newtheorem{lem}[thm]{Lemma}
\newtheorem{prop}[thm]{Proposition}
\newtheorem{cor}[thm]{Corollary}
\numberwithin{thm}{section}
\numberwithin{equation}{section}

\theoremstyle{definition}

\newtheorem*{conj}{Conjecture}

\newcommand{\rat}{\mathbb Q}
\newcommand{\real}{\mathbb R}
\newcommand{\com}{\mathbb C}
\newcommand{\alg}{\overline\rat}
\newcommand{\algt}{\alg^{\times}}

\newcommand{\intg}{\mathbb Z}
\newcommand{\nat}{\mathbb N}

\newcommand{\tor}{\mathrm{Tor}}

\newcommand{\xx}{{\bf x}}
\newcommand{\yy}{{\bf y}}
\newcommand{\comment}[1]{}

\title[Metric heights]{Metric heights on an Abelian group}
\author[C.L. Samuels]{Charles L. Samuels}
\address{Oklahoma City University, Department of Mathematics, 2501 N. Blackwelder, Oklahoma City, OK 73106, USA}

\begin{document}

\begin{abstract}
	Suppose $m(\alpha)$ denotes the Mahler measure of the non-zero algebraic number $\alpha$.
	For each positive real number $t$, the author studied a version $m_t(\alpha)$ of the Mahler measure that has the triangle inequality.
	The construction of $m_t$ is generic, and may be applied to a broader class of functions defined on any Abelian group $G$.
	We prove analogs of known results with an abstract function on $G$ in place of the Mahler measure.   In the process, we resolve an earlier
	open problem stated by the author regarding $m_t(\alpha)$.	
\end{abstract}

\maketitle

\section{Heights and their metric versions}

Suppose that $K$ is a number field and $v$ is a place of $K$ dividing the place $p$ of $\rat$.  Let $K_v$ and $\rat_p$ be their respective completions
so that $K_v$ is a finite extension of $\rat_p$.  We note the well-known fact that
\begin{equation*}
	\sum_{v\mid p} [K_v:\rat_p] = [K:\rat],
\end{equation*}
where the sum is taken over all places $v$ of $K$ dividing $p$.
Given $x\in K_v$, we define $\|x\|_v$ to be the unique extension of the $p$-adic absolute value on $\rat_p$ and set
\begin{equation} \label{NormalAbs}
	|x|_v = \|x\|_v ^{[K_v:\rat_p]/[K:\rat]}.
\end{equation}
If $\alpha\in K$, then $\alpha\in K_v$ for every place $v$, so we may define the {\it (logarithmic) Weil height} by
\begin{equation*}
	h(\alpha) = \sum_v \log^+ |\alpha|_v.
\end{equation*}
Due to our normalization of absolute values \eqref{NormalAbs}, this definition is independent of $K$, meaning that $h$ is well-defined
as a function on the multiplicative group $\algt$ of non-zero algebraic numbers.

It follows from Kronecker's Theorem that $h(\alpha) = 0$ if and only if $\alpha$ is a root of unity, and it can easily be verified that $h(\alpha^n) = |n|\cdot h(\alpha)$ for all
integers $n$.  In particular, we see that $h(\alpha) = h(\alpha^{-1})$.
A theorem of Northcott \cite{Northcott} asserts that, given a positive real number $D$, there are only finitely many algebraic numbers $\alpha$ with $\deg\alpha\leq D$ and 
$h(\alpha)\leq D$.

The Weil height is closely connected to a famous 1933 problem of D.H. Lehmer \cite{Lehmer}.  The {\it (logarithmic) Mahler measure} of a non-zero algebraic number $\alpha$ is
defined by 
\begin{equation} \label{MahlerMeasure}
	m(\alpha) = \deg\alpha \cdot h(\alpha).  
\end{equation}	
In attempting to construct large prime numbers, Lehmer came across the problem of determining whether there exists a sequence of
algebraic numbers $\{\alpha_n\}$, not roots of unity, such that $m(\alpha_n)$ tends to $0$ as $n\to\infty$.  This problem remains unresolved, although substantial evidence
suggests that no such sequence exists (see \cite{BDM, MossWeb, Schinzel, Smyth}, for instance).  This assertion is typically called Lehmer's conjecture.

\begin{conj}[Lehmer's Conjecture]
	There exists $c>0$ such that $m(\alpha) \geq c$ whenever $\alpha\in \algt$ is not a root of unity. 
\end{conj}

Dobrowolski \cite{Dobrowolski} provided the best known lower bound on $m(\alpha)$ in terms of $\deg\alpha$, while Voutier \cite{Voutier}
later gave a  version of this result with an effective constant.
Nevertheless, only little progress has been made on Lehmer's conjecture for an arbitrary algebraic number $\alpha$.

Dubickas and Smyth \cite{DubSmyth, DubSmyth2} were the first to study a modified version of the Mahler measure that has the triangle inequality.  They defined
the {\it metric Mahler measure} by
\begin{equation*}
	m_1(\alpha) = \inf\left\{ \sum_{n=1}^N m(\alpha_n): N\in\nat,\ \alpha_n\in \algt\ \alpha = \prod_{n=1}^N\alpha_n\right\},
\end{equation*}
so that the infimum is taken over all ways of writing $\alpha$ as a product of algebraic numbers.  It is easily verified that $m_1(\alpha\beta) \leq m_1(\alpha) + m_1(\beta)$,
and that $m_1$ is well-defined on $\algt/\tor(\algt)$.  It is further noted in \cite{DubSmyth2} that $m_1(\alpha) = 0$ if and only if $\alpha$ is a torsion point of $\algt$ and that
$m_1(\alpha) = m_1(\alpha^{-1})$ for all $\alpha\in\algt$.  These facts ensure that $(\alpha,\beta) \mapsto m_1(\alpha\beta^{-1})$ defines a metric on $\algt/\tor(\algt)$.  
This metric induces the discrete topology if and only if Lehmer's conjecture is true.

The author \cite{SamuelsCollection, SamuelsParametrized} further extended this definition to define the {\it $t$-metric Mahler measure}
\begin{equation} \label{TMetricMahler}
	m_t(\alpha) = \inf\left\{ \left( \sum_{n=1}^N m(\alpha_n)^t\right)^{1/t}:N\in\nat,\ \alpha_n\in \algt\ \alpha = \prod_{n=1}^N\alpha_n\right\}.
\end{equation}
In this context, we examined the functions $t\mapsto m_t(\alpha)$ for a fixed algebraic number $\alpha$.  It is shown, for example, that this is a continuous piecewise function
where each piece is an $L_p$ norm of a real vector.

Although \eqref{TMetricMahler} may be useful in studying Lehmer's problem, this construction applies in far greater generality.
The natural abstraction of \eqref{TMetricMahler} is examined by the author in \cite{SamuelsCollection}, although we only present the most basic results that are needed in 
studying the $t$-metric Mahler measures.  The goal of this article is to recover some results of \cite{SamuelsCollection} and \cite{SamuelsParametrized}
with an abstract height function in place of the Mahler measure.  In the process, we shall uncover new results regarding the $t$-metric Mahler measure, 
including the resolution of a problem posed in \cite{SamuelsParametrized}.  These results are reported in Section \ref{NewResults}.

Before we can state our main theorem, we must recall the basic definitions and results of \cite{SamuelsCollection}.
Let $G$ be a multiplicatively written Abelian group.  We say that $\phi:G \to [0,\infty)$ is a {\it (logarithmic) height} on $G$ if
\begin{enumerate}[(i)]
	\item $\phi(1) = 0$, and
	\item $\phi(\alpha) = \phi(\alpha^{-1})$ for all $\alpha\in G$.
\end{enumerate}
It is well-known that both the Weil height and the Mahler measure are heights on $\algt$.
If $t$ is a positive real number then we say that $\phi$ has the {\it $t$-triangle inequality} if
\begin{equation*}
	\phi(\alpha\beta) \leq \left (\phi(\alpha)^t + \phi(\beta)^t\right )^{1/t}
\end{equation*}
for all $\alpha,\beta\in G$.  We say that $\phi$ has the {\it $\infty$-triangle inequality} if
\begin{equation*}
	\phi(\alpha\beta) \leq \max\{\phi(\alpha),\phi(\beta)\}
\end{equation*}
for all $\alpha,\beta\in G$.  For appropriate $t$, we say that these functions are {\it $t$-metric heights}.   We observe that the $1$-triangle inequality is simply 
the classical triangle inequality while the $\infty$-triangle inequality is the strong triangle inequality.
If $s\geq t$ and $\phi$ is an $s$-metric height then it is also a $t$-metric height.
The metric height properties also yield some compatibility of $\phi$ with the group structure of $G$.

\begin{prop} \label{MetricProperties}
	If $\phi:G\to [0,\infty)$ is a $t$-metric height for some $t\in (0,\infty]$ then
	\begin{enumerate}[(i)]
		\item\label{Subgroup} $\phi^{-1}(0)$ is a subgroup of $G$.
		\item\label{WellDefined} $\phi(\zeta \alpha) = \phi(\alpha)$ for all $\alpha\in G$ and $\zeta\in \phi^{-1}(0)$.  That is, $\phi$ is well-defined on the quotient $G/\phi^{-1}(0)$.
		\item\label{FancyMetric} If $t\geq 1$, then the map $(\alpha,\beta)\mapsto \phi(\alpha\beta^{-1})$ defines a metric on $G/\phi^{-1}(0)$.
	\end{enumerate}
\end{prop}

If we are given an arbitrary height $\phi$ on $G$ and $t>0$, then following \eqref{TMetricMahler}, we obtain a natural $t$-metric height associated to $\phi$.  
Given any subset $S\subseteq G$ containing the identity $e$, we write
\begin{equation*}
	S^{\infty} = \{(\alpha_1,\alpha_2,\ldots): \alpha_n\in S,\ \alpha_n = e\ \mathrm{for\ all\ but\ finitely\ many}\ n\}.
\end{equation*}
If $S$ is a subgroup, then $S^\infty$ is also a group by applying the operation of $G$ coordinatewise.
Define the group homomorphism $\tau_G:G^{\infty} \to G$ by $\tau_G(\alpha_1,\alpha_2,\cdots) = \prod_{n=1}^\infty \alpha_n.$
For each point $\xx = (x_1,x_2,\ldots)\in \real^\infty$, we write the $L_t$ norm of $\xx$
\begin{equation*}
	\|\xx\|_t = \begin{cases}
		\left(\sum_{n=1}^\infty |x_n|^t\right)^{1/t} & \mathrm{if}\ t\in(0,\infty) \\
		\max_{n\geq 1}\{|x_n|\} & \mathrm{if}\ t=\infty
		\end{cases}
\end{equation*}
The {\it $t$-metric version of $\phi$} is the map $\phi_t:G\to[0,\infty)$ given by
\begin{equation*}
	\phi_t(\alpha) = \inf\left\{\|(\phi(\alpha_1),\phi(\alpha_2),\ldots)\|_t: (\alpha_1,\alpha_2,\ldots)\in G^\infty\ \mathrm{and}\ \tau_G(\alpha_1,\alpha_2,\ldots) = \alpha\right\}.
\end{equation*}
Alternatively, we could write
\begin{equation*}
	\phi_t(\alpha) = \inf\left\{\left(\sum_{n=1}^N \phi(\alpha_n)^t\right)^{1/t}: N\in\nat,\ \alpha_n\in G\ \mathrm{and}\ \alpha = \prod_{n=1}^N\alpha_n\right\}
\end{equation*}
for $t\in(0,\infty)$ and
\begin{equation*}
	\phi_\infty(\alpha) = \inf\left\{\max_{1\leq n\leq N}\{\phi(\alpha_n)\}: N\in\nat,\ \alpha_n\in G\ \mathrm{and}\ \alpha = \prod_{n=1}^N\alpha_n\right\}.
\end{equation*}
Among other things, we see that $\phi_t$ is indeed a $t$-metric height on $G$.

\begin{prop} \label{MetricConstruction}
  If $\phi:G\to [0,\infty)$ is a height on $G$ and $t\in (0,\infty]$ then
  \begin{enumerate}[(i)]
  \item\label{MetricHeightConversion} $\phi_t$ is a $t$-metric height on $G$ with $\phi_t\leq\phi$.
  \item\label{BestMetricHeight} If $\psi$ is a $t$-metric height with $\psi\leq\phi$ then $\psi\leq \phi_t$.
  \item\label{NoChangeMetric} $\phi = \phi_t$ if and only if $\phi$ is a $t$-metric height.  In particular, $(\phi_t)_t = \phi_t$.
  \item\label{Comparisons} If $s\in (0,t]$ then $\phi_s \geq \phi_t$.
  \end{enumerate}
\end{prop}

%We regard condition \eqref{BestMetricHeight} as significant because it shows that $\phi_t$ is the largest $t$-metric height smaller that $\phi$.

%In the case of the Mahler measure, it is clear that $m_t$ induces the discrete topology on $\algt$ if and only if Lehmer's conjecture is true.  
%Hence, a reasonable $\phi$-analog of Lehmer's problem is to ask whether $\phi_t$ induces the discrete topology for some $t\geq 1$.

\section{The function $t\mapsto \phi_t(\alpha)$} \label{NewResults}

For the remainder of this article, we shall assume that $\phi$ is a height on the Abelian group $G$ and that $\alpha$ is a fixed element of $G$.  All subsequent definitions
depend on these choices, although we will often suppress this dependency in our notation.
We say that a set $S\subseteq G$ containing the identity {\it replaces $G$ at $t$} if
\begin{equation*}
	\phi_t(\alpha)  = \inf\left\{\|(\phi(\alpha_1),\phi(\alpha_2),\ldots)\|_t: (\alpha_1,\alpha_2,\ldots)\in S^\infty\ \mathrm{and}\ \tau_G(\alpha_1,\alpha_2,\ldots) = \alpha\right\}.
\end{equation*}
In other words, we need only consider points in $S$ in the definition of $\phi_t(\alpha)$.

For a given height $\phi$, it is standard to ask whether the infimum in its definition is attained.  Before proceeding further, we give two equivalent conditions.

\begin{thm} \label{ReplaceableEquivalences}
	If $t\in(0,\infty]$ then the following conditions are equivalent.
	\begin{enumerate}[(i)]
		\item \label{InfAttained} The infimum in the definition of $\phi_t(\alpha)$ is attained.
		\item \label{FiniteReplaceable} There exists a finite set $R$ that replaces $G$ at $t$.
		\item \label{SemifiniteReplaceable} There exists a set $S$, with $\phi(S)$ finite, that replaces $G$ at $t$.
	\end{enumerate}
\end{thm}

We now wish to study the behavior of $f_{\phi,\alpha}:(0,\infty)\to [0,\infty)$, defined by $$f_{\phi,\alpha}(t) = \phi_t(\alpha),$$ on a given open interval $I\subseteq (0,\infty)$
as was done with the Mahler measure in \cite{SamuelsCollection, SamuelsParametrized}.   We say that $S\subseteq G$
{\it replaces $G$ uniformly on $I$} if $S$ replaces $G$ at all $t\in I$.  In this case, it is important to note that $S$ is independent of $t$.
We say that a subset $K\subseteq I$ is {\it uniform} if there exists a point $\xx\in\real^\infty$ such that $f_{\phi,\alpha}(t) = \|\xx\|_t$ for all $t\in K$.  Indeed, the uniform subintervals
of $I$ are the simplest to understand as there always exist $x_1,\ldots,x_N\in \real$ such that
\begin{equation*}
	f_{\phi,\alpha}(t) = \left(\sum_{n=1}^N |x_n|^t\right)^{1/t}
\end{equation*}
for all $t\in K$.  We say that $t$ is {\it standard} if there exists a uniform open interval $J\subseteq I$ containing $t$.  Otherwise, we say that $t$ is {\it exceptional}.
The author asked in \cite{SamuelsParametrized} whether the Mahler measure has only finitely many exceptional points.  We answer this question in the affirmative 
and prove several additional facts about heights in general.

\begin{thm} \label{UniformReplacements}
	Assume $I\subseteq (0,\infty)$ is an open interval such that the infimum in $\phi_t(\alpha)$ is attained for all $t\in I$.  If $G$ is a countable group
	then the following conditions are equivalent.
	\begin{enumerate}[(i)]
		\item\label{FiniteFunc} There exists a finite set $\mathcal X \subseteq \real^\infty$ such that
			$\phi_t(\alpha) =  \min\{ \|\xx\|_t: \xx\in \mathcal X\}$ for all $t\in I$.
		\item\label{FiniteExceptional} $I$ contains only finitely many exceptional points.
		\item\label{Finite} There exists a finite set $R$ that replaces $G$ uniformly on $I$.
		\item\label{SemiFinite} There exists a set $S$, with $\phi(S)$ finite, that replaces $G$ uniformly on $I$.
		
	\end{enumerate}
\end{thm}

In the case where $\phi$ is the Mahler measure, and where $I$ is a bounded interval, the majority of Theorem \ref{UniformReplacements} was established in \cite{SamuelsParametrized}.
Indeed, in this special case, we showed that \eqref{SemiFinite} $\implies$ \eqref{FiniteFunc} $\iff$ \eqref{FiniteExceptional}.
Nevertheless, Theorem \ref{UniformReplacements} is considerably more general than any earlier work because it requires neither the assumption that $\phi$ is the Mahler measure nor
the assumption that $I$ is bounded.
Moreover, since the Mahler measure is known to satisfy \eqref{SemiFinite}, we have now resolved the aforementioned question of \cite{SamuelsParametrized}.

\begin{cor} \label{ExcepCor}
	The Mahler measure $m$ on $\algt$ has only finitely many exceptional points in $(0,\infty)$.
\end{cor}

This means that $f_{m,\alpha}$ is a piecewise function, with finitely many pieces, where each piece is the $L_t$ norm of a vector with real entries.
Moreover, the infimum in $m_t(\alpha)$ is attained by a single point for all sufficiently large $t$.  In the case where $\alpha\in \intg$,
it follows from \cite{JankSamuels} that this point may be taken to be the vector having the prime factors of $\alpha$ as its entries.

Still considering the case $\phi = m$, the results of \cite{SamuelsCollection} show that
\begin{equation*}
	S_\alpha = \left\{\gamma\in \algt: \gamma^n\in K_\alpha\ \mbox{for some}\ n\in\nat,\ m(\gamma) \leq m(\alpha)\right\},
\end{equation*}
where $K_\alpha$ is the Galois closure of $\rat(\alpha)/\rat$, replaces $\algt$ uniformly on $(0,\infty)$.  However, it is well-known that $m(S_\alpha)$ is finite,
so Theorem \ref{UniformReplacements} implies the existence of a finite set $R_\alpha$ that replaces $\algt$ uniformly on $(0,\infty)$.  It remains open to 
determine such a set, although we suspect that
\begin{equation*}
	R_\alpha = \left\{\gamma\in \algt: \deg(\gamma)\leq \deg\alpha,\ m(\gamma) \leq m(\alpha)\right\}
\end{equation*}
satisfies this property.  The work of Jankauskas and the author \cite{JankSamuels} provides examples, however, in which 
$K_\alpha$ does not replace $\algt$ uniformly on $(0,\infty)$.

Returning to Theorem \ref{UniformReplacements} and taking an arbitrary $\phi$, 
it is important to note the necessity of our assumption that the infimum in $\phi_t(\alpha)$ is attained for every $t\in I$.  
Indeed, Theorem 1.6 of \cite{SamuelsCollection} asserts that
\begin{equation*}
	h_t(\alpha) = \begin{cases} h(\alpha) & \mathrm{if}\ t\leq 1 \\ 0 & \mathrm{if}\ t > 1, \end{cases}
\end{equation*}
where $h$ denotes the logarithmic Weil height on $\algt$.  The intervals $(0,1)$ and $(1,\infty)$ are both uniform by using $\xx = (h(\alpha),0,0,\ldots)$ and
$\xx = (0,0,0,\ldots)$, respectively, so $1$ is the only possible exceptional point.  However, $t\mapsto h_t(\alpha)$ is discontinuous on $(0,2)$ whenever $\alpha$ is not a root
of unity, so condition \eqref{FiniteFunc}
does not hold on this interval.  Theorem \ref{UniformReplacements} does not apply because the infimum in $h_t(\alpha)$ is not attained for any $t > 1$.  
Nevertheless, the assumption \eqref{SemiFinite} would imply that the infimum in $\phi_t(\alpha)$ is attained for all $t\in I$ from Theorem \ref{ReplaceableEquivalences}.

The assumption that $G$ is countable is needed only to prove that \eqref{FiniteExceptional} $\implies$ \eqref{Finite}.
If an interval $I$ is uniform, then $\phi_t(\alpha) = \|\xx\|_t$ for all $t\in I$.   It seems plausible, however, that $\xx$ does not arise from a point
that attains the infimum in $\phi_t(\alpha)$.  This concern can be resolved if, for example, $G$ is countable, although we do not know whether this assumption is necessary.

\section{Proofs}

Our first lemma is the primary component in the proof of Theorem \ref{ReplaceableEquivalences} and will also be used in the proof of Theorem \ref{UniformReplacements}.

\begin{lem} \label{Finiteness}
	Suppose $J\subseteq (0,\infty)$ is bounded.  If there exists a set $S\subseteq G$, with $\phi(S)$ finite, such that $S$ replaces $G$ uniformly on $J$, then there exists
	a finite set $D\subseteq \tau_G^{-1}(\alpha)\cap S^\infty$ such that
	\begin{equation*}
		\phi_t(\alpha) = \min\left\{\left(\sum_{n=1}^\infty \phi(\alpha_n)^t\right)^{1/t}: (\alpha_1,\alpha_2\ldots)\in D\right\}
	\end{equation*}
	for all $t\in J$.  In particular, there exists a finite set $\mathcal X \subseteq \real^\infty$ such that $\phi_t(\alpha) =  \min\{ \|\xx\|_t: \xx\in \mathcal X\}$ for all $t\in J$.
\end{lem}
\begin{proof}
	We know that that $\phi_t(\alpha)$ is the infimum of 
	\begin{equation} \label{InfimumQuantity}
		\left( \sum_{n=1}^\infty \phi(\alpha_n)^t\right)^{1/t}
	\end{equation}
	over all points $(\alpha_1,\alpha_2,\ldots) \in G^\infty$ satisfying the following three conditions:
	\begin{enumerate}[(a)]
		\item\label{Equals} $\alpha =\prod_{n=1}^\infty \alpha_n$.
		\item\label{InS} $\alpha_n\in S$ for all $n$.
		\item\label{Bounded} $\left( \sum_{n=1}^\infty \phi(\alpha_n)^t\right)^{1/t} \leq \phi(\alpha)$.
	\end{enumerate}
	If $(\alpha_1,\alpha_2,\ldots)\in G^\infty$ is a point satisfying these conditions,
	let $B$ be the number of its coordinates that do not belong to $\phi^{-1}(0)$.  Also set
	$\delta = \inf \phi(S\setminus\phi^{-1}(0))$.  Since $\phi(S)$ is finite, we know that $\delta >0$.  By property \eqref{Bounded}, we see that
	\begin{equation*}
		\delta^t\cdot B \leq  \sum_{n=1}^\infty \phi(\alpha_n)^t \leq \phi(\alpha)^t,
	\end{equation*}
	so it follows that $B \leq  \phi(\alpha)^u/\delta^u$, where $u$ is any upper bound for $J$.
	Therefore, at most $\phi(\alpha)^u/\delta^u$ terms $\phi(\alpha_n)$ in \eqref{InfimumQuantity} can be non-zero and the result follows.
\end{proof}

In view of this lemma, the proof of Theorem \ref{ReplaceableEquivalences} is essentially finished.  Indeed, one obtains \eqref{SemifiniteReplaceable}
$\implies$ \eqref{InfAttained} immediately from the lemma by taking $J=\{t\}$, while the other implications of the theorem are obvious.

We shall now proceed with the proof of Theorem \ref{UniformReplacements} which will require three additional lemmas, the first of which is a standard trick from complex analysis.

\begin{lem} \label{LimitEqual}
	If $\xx, \yy \in \real^\infty$ and $\|\xx\|_t = \|\yy\|_t$ for all $t$ on a set having a limit point in $\real$, then $\|\xx\|_t = \|\yy\|_t$ for all $t\in(0,\infty)$.
\end{lem}
\begin{proof}
	Let $\xx = (x_1,x_2,\ldots,x_N,0,0,\ldots)$, $\yy = (y_1,y_2,\ldots,y_M,0,0,\ldots)$, $g_\xx(t) = \|\xx\|_t^t$ and $g_\yy(t) = \|\yy\|_t^t$.   Therefore,
	\begin{equation*}
		g_\xx(z) = \sum_{n=1}^N |x_n|^z \quad\mathrm{and}\quad g_\yy(t)  = \sum_{m=1}^M |y_m|^z
	\end{equation*}
	are entire functions.  Hence, if they agree on a set having a limit point in $\com$, then they are equal in $\com$.

\end{proof}

We also must study the behavior of intervals in which every point is standard.  While every point in such an interval is guaranteed only to have a uniform open neighborhood,
it turns out that this neighborhood may be taken to be the interval itself.

\begin{lem} \label{CoverLemma}
  Suppose $0\leq a < b\leq \infty$.  Then $(a,b)$ is uniform if and only if every point in $(a,b)$ is standard.
\end{lem}
\begin{proof}
	If $(a,b)$ is uniform, it is obvious that every point in $(a,b)$ is standard.  So we will assume that every point in $(a,b)$ is standard and that $(a,b)$ is not uniform.
	Let $t_0\in (a,b)$ so there exists $\varepsilon > 0$ and $\xx\in \real$ such that $\phi_t(\alpha) = \|\xx\|_t$
	for all $t\in (t_0 - \varepsilon,t_0+\varepsilon)$.  We have assumed that $(a,b)$ is not uniform, so there must exist $s_0\in(a,b)$ such that
	$\phi_{s_0}(\alpha)\ne \|\xx\|_{s_0}$.  Clearly $s_0 \ne t_0$.
	
	Assume without loss of generality that $s_0 > t_0$ and let
	\begin{equation*} \label{uInfimum}
		u = \inf\{t\in [t_0,b): \phi_t(\alpha) \ne \|\xx\|_t\}.
	\end{equation*}
	We clearly have that $u\in[t_0+\varepsilon,s_0] \subseteq (a,b)$, meaning, in particular, that $u$ is standard.  Therefore, there exists a neighborhood 
	$(a_0,b_0)\subseteq (t_0,b)$ of $u$ and $\yy\in \real^{\infty}$ such that
	\begin{equation} \label{S0Point}
		\phi_t(\alpha)= \|\yy\|_t
	\end{equation}
	for all $t\in (a_0,b_0)$.  However, by definition of $u$, we also know that $\phi_t(\alpha) = \|\xx\|_t$ for all $t\in (a_0,u)$.  Therefore, $\|\yy\|_t$ and $\|\xx\|_t$ agree on 
	a set having a limit point in $\real$, and we may apply Lemma \ref{LimitEqual} to find that they agree on $(a,b)$.  The definition of $u$ further implies the
	existence of a
	point $s\in [u,b_0)$ such that $\phi_s(\alpha) \ne \|\xx\|_s$.  Combining this with \eqref{S0Point}, we see that $\|\xx\|_s \ne \|\yy\|_s$, a contradiction.
			
	%Now assume that $s_0 < t_0$ and let
	%\begin{equation*}
	%	v = \sup\{t\in (a,t_0]: \phi_t(\alpha) \ne \|\xx\|_t\}.
	%\end{equation*}
	%We clearly have that $v\in[s_0,t_0-\varepsilon]$ so that $v$ is standard.  Therefore, there exists a neighborhood $(a_1,b_1) \subseteq (a,t_0)$ of $v$ and
	%$\yy\in \real^{\infty}$ such that
	%\begin{equation*}
	%	\phi_t(\alpha) = \|\yy\|_t
	%\end{equation*}
	%for all $t\in (a_1,b_1)$.  However, by definition of $v$, we also know that $\phi_t(\alpha) = \|\xx\|_t$ for all $t\in (v,b_1)$.  Therefore, $\|\yy\|_t$ and $\|\xx\|_t$ agree on 
	%a set having a limit point in $\real$, and we may apply Lemma \ref{LimitEqual} to find that they agree on $(a,b)$.  Again applying the definition of $v$, there exists a point
	%$s\in (a_1,v]$ such that $\phi_s(\alpha) \ne \|\xx\|_s$.  Along with \eqref{S0Point}, we get that $\|\xx\|_s \ne \|\yy\|_s$, yielding another contradiction.
\end{proof}
	
For a uniform interval $I$, we always know there exists $\xx\in\real^\infty$ such that $\phi_t(\alpha) = \|\xx\|_t$ for all $t\in I$.  Nevertheless, it is possible
that $\xx$ cannot be chosen to arise from a point that attains the infimum in $\phi_t(\alpha)$.  The following lemma provides an adequate resolution to this problem in the 
case where $G$ is countable.
	
\begin{lem} \label{Countable}
	Suppose $G$ is a countable group and $K\subseteq (0,\infty)$ is an uncountable subset.  Assume that the infimum in $\phi_t(\alpha)$ is attained for all $t\in K$.
	Then $K$ is uniform if and only if there exists $(\alpha_1,\alpha_2,\ldots)\in G^\infty$ such that
	\begin{equation} \label{SpecialUniform}
		\alpha = \prod_{n=1}^\infty \alpha_n\quad\mathrm{and}\quad \phi_t(\alpha) = \left(\sum_{n=1}^\infty \phi(\alpha_n)^t\right)^{1/t}
	\end{equation}
	for all $t\in K$.
\end{lem}
\begin{proof}
	Assuming \eqref{SpecialUniform}, it is obvious that $K$ is uniform.  Hence, we assume that $K$ is uniform and let $\xx\in\real^\infty$ be such that
	\begin{equation} \label{UniformAssumption}
		\phi_t(\alpha) = \|\xx\|_t
	\end{equation}
	for all $t\in K$.  We have assumed the infimum in $\phi_t(\alpha)$ is attained for all $t\in K$.   Hence, for each such $t$, we may select 
	$(\alpha_{t,1},\alpha_{t,2},\ldots)\in G^\infty$ such that
	\begin{equation*}
		\alpha = \prod_{n=1}^\infty \alpha_{t,n}\quad\mathrm{and}\quad \phi_t(\alpha) = \left(\sum_{n=1}^\infty \phi(\alpha_{t,n})^t\right)^{1/t}.
	\end{equation*}
	Since $K$ is uncountable, 
	$t\mapsto (\alpha_{t,1},\alpha_{t,2},\ldots)$ maps an uncountable set to a countable set.  In particular, this map must be constant on an 
	uncountable subset $J\subseteq K$.  Hence, there exists $(\alpha_1,\alpha_2,\ldots)\in G^\infty$ such that
	\begin{equation*}
		\alpha = \prod_{n=1}^\infty \alpha_n\quad\mathrm{and}\quad \phi_t(\alpha) = \left(\sum_{n=1}^\infty \phi(\alpha_n)^t\right)^{1/t},
	\end{equation*}
	for all $t\in J$.  Applying \eqref{UniformAssumption}, we have that
	\begin{equation} \label{EqualInJ}
		\left(\sum_{n=1}^\infty \phi(\alpha_n)^t\right)^{1/t} =  \|\xx\|_t
	\end{equation}
	for all $t\in J$.  Since $J$ is uncountable, it must contain a limit point in $[0,\infty)$, and it follows from Lemma \ref{LimitEqual} that \eqref{EqualInJ}
	holds for all $t\in K$.  The lemma now follows from \eqref{UniformAssumption}.
\end{proof}
	
Before proceeding with the proof of Theorem \ref{UniformReplacements}, we provide a definition that will simplify the proof's language.
If $\mathcal X \subseteq \real^\infty$, we say that $s\in (0,\infty)$ is an {\it intersection point of $\mathcal X$} if there exist $\xx, \yy\in \mathcal X$ such that 
$\|\xx\|_s = \|\yy\|_s$ but $t\mapsto \|\xx\|_t$ is not the same function as $t\mapsto \|\yy\|_t$.

\begin{proof}[Proof of Theorem \ref{UniformReplacements}]
	We begin by proving that \eqref{FiniteFunc} $\implies$ \eqref{FiniteExceptional}.  We first show that $\mathcal X$ has only finitely many intersection points.
	To see this, assume that
	\begin{equation*}
		(x_1,x_2,\ldots,x_N,0,0,\ldots), (y_1,y_2,\ldots,y_M,0,0,\ldots) \in\mathcal X,
	\end{equation*}
	with $x_n, y_m\ne 0$, are distinct elements and define
	\begin{equation*}
		F(t) = \|\xx\|_t^t - \|\yy\|_t^t =  \sum_{n=1}^N |x_n|^t - \sum_{m=1}^M |y_n|^t.
	\end{equation*}
	We may assume without loss of generality that
	\begin{equation*}
		F(t) = \sum_{k=1}^K a_kb_k^t
	\end{equation*}
	for some positive integer $K$ and nonzero real numbers $a_k$ and $b_k$ with $b_1> \cdots > b_K > 0$.  Then it follows that
	\begin{equation*}
		\frac{F(t)}{a_1b_1^t} = 1 + \sum_{k=2}^K \frac{a_k}{a_1} \left(\frac{b_k}{b_1}\right)^t
	\end{equation*}
	which tends to $1$ as $t \to\infty$.  Since $a_1b_1^t$ has no zeros, all zeros of $F(t)$ must lie in a bounded subset of $(0,\infty)$.  
	Viewing $t$ as a complex variable, $F(t)$ is an entire function which is not identically $0$, so it cannot have infinitely many zeros in a bounded set.
	We conclude that there are only finitely many points $t$ such that $\|\xx\|_t = \|\yy\|_t$.
	Since $\mathcal X$ is finite, it can have only finitely many intersection points.
	
	It is now enough to show that every exceptional point is also an intersection point of $\mathcal X$.  We will prove the contrapositive of this statement, so 
	assume that $t\in I$ is not an intersection point of $\mathcal X$.
	Since there are only finitely many intersection points, we know there exists an open interval $J\subseteq I$ containing $t$ and having no intersection points of $\mathcal X$.
	
	Now assume that $J$ fails to be uniform and fix $t\in J$.  By our assumption, we know there exists $\xx\in \mathcal X$ such that $\phi_t(\alpha) = \|\xx\|_t$.
	There must also exist $s\in J$ such that $\phi_s(\alpha) \ne \|\xx\|_s$.  Since $\xx\in \mathcal X$, we certainly have that $\phi_s(\alpha) < \|\xx\|_s$.
	On the other hand, there must exist $\yy\in\mathcal X$ such that $\phi_s(\alpha) = \|\yy\|_s$, and we note that $\phi_t(\alpha) \leq \|\yy\|_t$.  Therefore, we have shown that
	\begin{equation*}
		\|\xx\|_t \leq \|\yy\|_t\quad\mathrm{and}\quad \|\xx\|_s > \|\yy\|_s.
	\end{equation*}
	By the Intermediate Value Theorem, there exists a point $r$ between $s$ and $t$ such that $\|\xx\|_r = \|\yy\|_r$.  This means that $J$ contains
	an intersection point, contradicting our assumption that $J$ contains none.  Therefore, we conclude that $J$ is uniform,
	implying that $t$ is standard.  Indeed, we have now shown that every exceptional point is also an intersection point of $\mathcal X$.

	We now prove that \eqref{FiniteExceptional} $\implies $ \eqref{Finite}.  To see this, assume that $I = (t_0,t_M)$ and that the exceptional points in $I$ are given by
	\begin{equation*}
		t_1 < t_2 < \cdots < t_{M-1}.
	\end{equation*}
	For each integer $m\in [1,M]$, it follows that $(t_{m-1},t_m)$ contains only standard points.  Applying Lemma \ref{CoverLemma}, we conclude that all of these intervals
	are uniform.
	
	Now we apply Lemma \ref{Countable} with $(t_{m-1},t_m)$ in place of $K$.  For each integer $m\in[1,M]$, there must exist 
	$(\alpha_{m,1},\alpha_{m,2},\ldots)\in G^\infty$ such that
	\begin{equation} \label{MReps}
		\alpha = \prod_{n=1}^\infty \alpha_{m,n}\quad\mathrm{and}\quad \phi_t(\alpha) = \left(\sum_{n=1}^\infty \phi(\alpha_{m,n})^t\right)^{1/t},
	\end{equation}
	for all $t\in (t_{m-1},t_m)$.  Moreover, for each integer $m\in [1,M)$, we select $(\beta_{m,1},\beta_{m,2},\ldots)\in G^\infty$ such that
	\begin{equation*} \label{NonMReps}
		\alpha = \prod_{n=1}^\infty \beta_{m,n}\quad\mathrm{and}\quad \phi_{t_m}(\alpha) = \left(\sum_{n=1}^\infty \phi(\beta_{m,n})^{t_m}\right)^{1/{t_m}}.
	\end{equation*}
	Now let
	\begin{equation*}
		R = \left\{ \alpha_{m,n}: 1\leq m\leq M,\ n\in\nat\right\} \cup \left\{ \beta_{m,n}: 1\leq m < M,\ n\in\nat\right\}
	\end{equation*}
	and note that $R$ is clearly finite.
	
	To see that $R$ replaces $G$ uniformly on $I$, we must assume that $t\in I$.  If $t\in (t_{m-1},t_m)$ for some $m$, then
	\begin{align*}
		\phi_t(\alpha) & \leq \inf\left\{ \|(\phi(\alpha_1),\phi(\alpha_2),\ldots)\|_t: (\alpha_1,\alpha_2,\ldots)\in R^\infty,\ \alpha = \prod_{n=1}^\infty \alpha_n\right\} \\
			& \leq \|(\phi(\alpha_{m,1}),\phi(\alpha_{m,2}),\ldots)\|_t \\
			& = \phi_t(\alpha),
	\end{align*}
	where the last equality is exactly the right hand side of \eqref{MReps}.  Given any $m$, we have now shown that
	\begin{equation} \label{ReplaceableEquality}
		\phi_t(\alpha) = \inf\left\{ \|(\phi(\alpha_1),\phi(\alpha_2),\ldots)\|_t: (\alpha_1,\alpha_2,\ldots)\in R^\infty,\ \alpha = \prod_{n=1}^\infty \alpha_n\right\}
	\end{equation}
	for all $t\in (t_{m-1},t_m)$.  Moreover, the same argument shows that \eqref{ReplaceableEquality} holds when $t= t_m$ for some $1\leq m < M$.  
	%Since
	%\begin{equation*}
	%	I = \bigcup_{m=1}^M (t_{m-1},t_m) \cup \bigcup_{m=1}^{M-1} \{ t_m\},
	%\end{equation*}
	We finally obtain that \eqref{ReplaceableEquality} holds for all $t\in I$ as required.
	
	It is immediately obvious that \eqref{Finite} $\implies$ \eqref{SemiFinite}, so to complete the proof, we must now show that \eqref{SemiFinite} $\implies$ \eqref{FiniteFunc}.
	Lemma \ref{Finiteness} establishes this implication in the case where $I$ is a bounded interval, so we shall assume that $I = (a,\infty)$ for some $a\geq 0$.
	Moreover, this lemma enables us to assume the existence of a set $A \subseteq \phi(S)^\infty$ such that
	\begin{equation*}
		\phi_t(\alpha) = \min\left\{\|\xx\|_t:\xx\in A\right\}
	\end{equation*}
	for all $t\in (a,\infty)$.  Without loss of generality, we may assume that 
	\begin{equation} \label{Falling}
		x_1\geq x_2\geq x_3\geq \ldots
	\end{equation}
	for all $(x_1,x_2,\ldots)\in A$.  Then we define a recursive sequence of sets $A_k\subseteq A$ in the following way.
	\begin{enumerate}[(i)]
		\item Let $A_0 = A$.
		\item Given $A_k$, let $M_{k+1} = \min\{ x_{k+1}:(x_1,x_2,\ldots)\in A_k\}$.  Observe that $M_{k+1}$ exists and is non-negative because $\phi(S)$ is a finite set of
		non-negative numbers and $x_{k+1} \in \phi(S)$.  Now let $A_{k+1} = \left\{ (x_1,x_2,\ldots) \in A_k: x_{k+1} = M_{k+1}\right\}$.
	
	\end{enumerate}
	It is immediately clear that all of these sets are nonempty and $A_{k+1}\subseteq A_k$ for all $k$.
	
	We claim that for every $k\geq 0$, there exists $s_k\in [a,\infty)$ such that
	\begin{equation} \label{FiniteMin}
		\phi_t(\alpha) = \min\left\{ \|\xx\|_t: \xx\in A_k\right\}
	\end{equation}
	for all $t\in (s_k,\infty)$.  We prove this assertion using induction on $k$, and we obtain the base case easily by taking $s_0 = a$.
	
	Now assume that there exists $s_k\in [a,\infty)$ such that \eqref{FiniteMin} holds for all $t\in (s_k,\infty)$.  If $A_{k+1} = A_k$, then we take $s_{k+1} = s_k$
	and the result follows.  Assuming otherwise, we may define
	\begin{equation*}
		M_{k+1}' = \min\left\{x_{k+1}: (x_1,x_2,\ldots) \in A_k \setminus A_{k+1}\right\}.
	\end{equation*}
	As before, we know that this minimum exists because $x_{k+1}$ belongs to the finite set $\phi(S)$ for all $(x_1,x_2,\ldots) \in A_k \setminus A_{k+1}$.
	Furthermore, we plainly have that $M_{k+1}' > M_{k+1}$.  If $t > s_k$ then \eqref{FiniteMin} implies that
	\begin{equation*}
		\phi_t(\alpha)^t = \min\left\{ \|\xx\|_t^t: \xx\in A_k\right\} = \min\left\{ \sum_{n=1}^\infty x_n^t: (x_1,x_2,\ldots) \in A_k\right\}.
	\end{equation*}
	Then we conclude that
	%\begin{equation*}
	%	\phi_t(\alpha)^t - \sum_{n=1}^k M_n^t =  \min\left\{ \sum_{n=1}^\infty x_n^t - \sum_{n=1}^k M_n^t: (x_1,x_2,\ldots) \in A_k\right\}.
	%\end{equation*}
	%For $(x_1,x_2,\ldots) \in A_k$, we know that $x_n = M_n$ for all $1\leq n\leq k$, which yields
	\begin{equation*}
		\phi_t(\alpha)^t - \sum_{n=1}^k M_n^t =  \min\left\{ \sum_{n=k+1}^\infty x_n^t: (x_1,x_2,\ldots) \in A_k\right\}.
	\end{equation*}
	Since $A_{k+1}$ is non-empty, there exists a point $(y_1,y_2,\ldots) \in A_{k+1} \subseteq A_k$, and we obtain
	\begin{equation*}
		\phi_t(\alpha)^t - \sum_{n=1}^k M_n^t  \leq \sum_{n=k+1}^\infty y_n^t
	\end{equation*}
	for all $t\in (s_k,\infty)$.  Therefore,
	\begin{equation*}
		\limsup_{t\to\infty} \left(\phi_t(\alpha)^t - \sum_{n=1}^k M_n^t\right)^{1/t} \leq \limsup_{t\to\infty} \left( \sum_{n=k+1}^\infty y_n^t\right)^{1/t} = y_{k+1} = M_{k+1}.
	\end{equation*}
	We have already noted that $M_{k+1} < M_{k+1}'$, which leads to
	\begin{equation*}
		\limsup_{t\to\infty} \left(\phi_t(\alpha)^t - \sum_{n=1}^k M_n^t\right)^{1/t} < M_{k+1}'.
	\end{equation*}
	Hence, there exists $s_{k+1}\in [s_k,\infty)$ such that  $\left(\phi_t(\alpha)^t - \sum_{n=1}^k M_n^t\right)^{1/t} < M_{k+1}'$ for all $t\in (s_{k+1},\infty)$.
	
	For any point $(x_1,x_2,\ldots) \in A_{k} \setminus A_{k+1}$,  we have that $M_{k+1}' \leq x_{k+1}$ so that
	\begin{equation*}
		 \left(\phi_t(\alpha)^t - \sum_{n=1}^k M_n^t\right)^{1/t} < x_{k+1} \leq \left( \sum_{n=k+1}^\infty x_n^t\right)^{1/t}.
	\end{equation*}
	But $(x_1,x_2,\ldots)\in A_k$, which means that $M_n = x_n$ for all $1\leq n\leq k$, and it follows that
	\begin{equation*}
		\phi_t(\alpha)^t < \sum_{n=1}^k M_n^t + \sum_{n=k+1}^\infty x_n^t = \sum_{n=1}^\infty x_n^t.
	\end{equation*}
	Then using \eqref{FiniteMin}, we have now shown that
	\begin{equation*}
		\min\left\{\|\xx\|_t:\xx\in A_k\right\} < \left( \sum_{n=1}^\infty x_n^t\right)^{1/t}
	\end{equation*}
	for all $(x_1,x_2,\ldots)\in A_k\setminus A_{k+1}$.  Finally, we obtain that
	\begin{equation*}
		\phi_t(\alpha) = \min\left\{\|\xx\|_t:\xx\in A_k\right\} = \min\left\{\|\xx\|_t:\xx\in A_{k+1}\right\}
	\end{equation*}
	for all $t\in (s_{k+1},\infty)$ verifying the inductive step and completing the proof of our claim.
	
	The sequence $M_1,M_2,\ldots $ is non-increasing, so by our assumption that $\phi(S)$ is finite, we know that this sequence is eventually constant.
	Suppose $k\in \nat$ is such that $M_n = M_{n+1}$ for all $n\geq k$.  It is straightforward to verify by induction that there must exist an element $(x_1,x_2,\ldots ) \in A_{k}$ such that
	$x_n = M_n$ for all $n\leq k$.  We cannot have $x_{k+1} < M_{k+1}$ because this would contradict the definition of $M_{k+1}$.  We also cannot have 
	$x_{k+1} > M_{k+1}$ because then $x_{k+1} > M_{k} = x_{k}$ which would contradict \eqref{Falling}.  Therefore, we must have $x_{k + 1} = M_{k + 1}$.
	By induction, we conclude that $x_n = M_n$ for all $n\geq k$.  
	
	We have now shown that $x_n = M_n$ for all $n$, which implies immediately that
	\begin{equation*}
		{\bf M} = (M_1,M_2,\ldots) \in A_k.
	\end{equation*}
	Moreover, we have assumed that $M_n = M_{n+1}$ for all $n\geq k$,
	which means that $x_n = x_{n+1}$ for all $n\geq k$.  Since the sequence $x_1,x_2,\ldots$ must be eventually zero, we conclude that $x_n = M_n = 0$
	for all $n\geq k$.  It now follows that ${\bf M}$ is the only element in $A_k$, implying that
	$\phi_t(\alpha) = \|{\bf M}\|_t$ for all $t\in (s_k,\infty)$.  By Lemma \ref{Finiteness}, there exists a finite set $\mathcal X_0$
	such that $\phi_t(\alpha) = \{\|\xx\|_t:\xx\in \mathcal X_0\}$ for all $t\in (a,s_k + 1)$ and we complete the proof by taking $\mathcal X = \mathcal X_0 \cup \{{\bf M}\}$.
	
\end{proof}

If a set $A$ can be determined, then our proof reveals a finite process for finding ${\bf M}$, even though $A$ 
is possibly infinite.  We are still unaware of a method by which to estimate the number of steps needed to find ${\bf M}$.

\end{document}